\documentclass[11pt, twoside]{article}
\usepackage[a4paper,left=2.5cm,right=2.5cm,top=2.5cm,bottom=2.5cm]{geometry} 
\usepackage{etoolbox}
\usepackage{amsmath,amssymb,amsthm,amsfonts,dsfont,mathrsfs}
\usepackage{physics}
\usepackage{bbm}
\usepackage{graphicx, fancyhdr}
\usepackage{algorithm}
\usepackage[noend]{algpseudocode}
\usepackage{caption}
\usepackage[margin=0.1cm]{subcaption}
\usepackage{footnote}
\usepackage[dvipsnames]{xcolor}
\usepackage{comment}
\usepackage{enumitem}

\usepackage{thmtools}
\usepackage{thm-restate}
\usepackage{lipsum}

\usepackage[utf8]{inputenc}

\usepackage{hyperref} 
\usepackage[nameinlink]{cleveref}
\let\cref = \Cref

\hypersetup{colorlinks=true,
            citecolor=ForestGreen,
            linkcolor=ForestGreen,    
            urlcolor=Brown}

\usepackage[backend=biber,style=numeric,maxnames=10]{biblatex}
\AtEveryBibitem{
    \clearlist{address}
    \clearfield{date}
    \clearfield{eprint}
    \clearfield{isbn}
    \clearfield{issn}
    \clearlist{location}
    \clearfield{month}
    \clearfield{series}
    \clearfield{doi}
    \ifentrytype{book}{}{
        \clearlist{publisher}
        \clearname{editor}
    }
}

\setlist[enumerate]{label=\upshape(\roman*), topsep=1ex, itemsep=0.5ex, parsep=0pt, leftmargin=8ex}

\newtheoremstyle{mystyle}
  {1.2em}
  {1.2em}
  {\itshape}
  {}
  {\bfseries}
  {.}
  {.5em}
  {}

\newtheorem{theorem}{Theorem}[section]
\newtheorem{lemma}[theorem]{Lemma}

\newtheorem{corollary}[theorem]{Corollary}

\newtheorem{example}[theorem]{Example}
\newtheorem{remark}[theorem]{Remark}
\newtheorem{definition}[theorem]{Definition}

\let\ga=\alpha \let\gb=\beta \let\gc=\gamma \let\gd=\delta 
            
 \let\go=\omega

\newcommand{\cM}{\mathcal{M}}\newcommand{\cN}{\mathcal{N}}\newcommand{\cO}{\mathcal{O}}

\newcommand{\bs}[1]{\boldsymbol{#1}}

\DeclareMathOperator{\E}{\mathds{E}}
\DeclareMathOperator{\R}{\mathbb{R}}

\DeclareMathOperator{\Z}{\mathbb{Z}}

\DeclareMathOperator{\pr}{\mathds{P}}

\DeclareMathOperator{\im}{Im}

\renewcommand{\v}[1]{\bs{#1}} 

\renewcommand{\t}[1]{\bs{\mathcal{#1}}} 

\newcommand{\m}[1]{\bs{#1}}

\def\a{\v a}

\def\u{\v u}

\def\x{\v x}
\def\y{\v y}

\def\A{\m A}

\def\D{\m D}
\def\G{\m G}
\def\H{\m H}

\def\P{\m P}

\def\S{\m S}

\renewcommand{\leq}{\leqslant} 
\renewcommand{\geq}{\geqslant} 

\let\it = \textit
\let\bf = \textbf

\let\sset = \subseteq

\let\inter = \cap
\let\Union = \bigcup

\newcommand{\set}[1]{\left\{#1\right\}}

\let\eps = \varepsilon

\let\wt = \widetilde
\let\wh = \widehat
\let\wb = \overline

\let\cp\relax
\DeclareMathOperator{\cp}{CP}

\newcommand{\ind}[1]{\mathds{1}_{#1}}

\DeclareMathOperator{\erip}{eRIP}
\DeclareMathOperator{\argmed}{argmed}
\DeclareMathOperator{\med}{med}
\DeclareMathOperator{\relu}{ReLU}
\DeclareMathOperator{\poly}{poly}
\def\pio{\pi_\circ}
\newcommand{\zcl}[1]{\overline{#1}^z}

\title{Norming Sets for Tensor and Polynomial Sketching}
\author{Yifan Zhang\thanks{Oden Institute, University of Texas at Austin (\href{mailto:yf.zhang@utexas.edu}{yf.zhang@utexas.edu}).}\and Joe Kileel\thanks{Department of Mathematics and Oden Institute, University of Texas at Austin (\href{mailto:jkileel@math.utexas.edu}{jkileel@math.utexas.edu}).}}
\date{}

\begin{document}

\maketitle

\begin{abstract}
    \noindent
    This paper develops the sketching (i.e., randomized dimension reduction) theory for real algebraic varieties and images of polynomial maps, including, e.g., the set of low rank tensors and tensor networks. Through the lens of norming sets, we provide a framework for controlling the sketching dimension for \it{any} sketch operator used to embed said sets, including sub-Gaussian, fast Johnson-Lindenstrauss, and tensor structured sketch operators. Leveraging norming set theory, we propose a new sketching method called the median sketch.  It embeds such a set $V$ using only $\wt\cO(\dim V)$ {tensor structured} or {sparse} linear measurements. 
\end{abstract}

\vspace*{0.5ex}
{\small
\noindent\bf{Key words:} randomized sketching, dimension reduction, Johnson-Lindenstrauss transform, real algebraic variety, polynomial image, tensor, norming set, median of means \hfill\\

}
\hfill\\[1.5em]

\section{Introduction}

Dimension reduction though random maps, also known as \it{sketching}, has received massive attention in the era of big data, as the curse of dimensionality represents a major bottleneck in many applications.
Given a set of data $\set{\x_1,\ldots,\x_p}$ in $\R^N$,
the well known Johnson-Lindenstrauss transform (JLT) \cite{johnson1984extensions} is a randomized linear transformation $\S: \R^N \rightarrow \R^m$ that is an approximate isometry on the dataset with respect to the Euclidean/Frobenius norm $\| \cdot \| = \| \cdot \|_2$:
\begin{equation*}
    \|\S(\x_i - \x_j)\| \approx \|\x_i - \x_j\|~~\text{for all } i, j,
\end{equation*}
while greatly reducing the dimension of the data points from $N$ to $m = \cO(\log p)$.
We call $m$ the \textit{sketching dimension}.
Since the JLT approximately preserves the geometry of the dataset, its compression power can lead to a speedup in downstream tasks such as clustering or classification.
Therefore, substantial research efforts have been invested into improving the design of a JLT matrix $\S$ while maintaining a low sketching dimension.
Examples include the standard sub-Gaussian sketch, fast JL transform (FJLT) \cite{ailon2006approximate} based on FFT-type  transforms, and sparse sketch operators (OSNAP) \cite{nelson2013osnap}, etc.  The latter two sketch operators require less space to store and can be applied to datasets using fewer flops, and are therefore advantageous computationally.

Besides finite datasets, researchers have also considered embedding infinite subsets.
The problem that has received the most attention is sketching (i.e., embedding) a linear subspace in $\R^N$, for linear algebra is the core in many applications.
Specifically, given a linear subspace $U \subseteq \mathbb{R}^N$, the sketch operator $\S$ has to be such that for all $\x \in U$ we have
$\|\S\x\| \approx \|\x\|$. 
Such a sketch operator is called a subspace embedding (SE) of $U$.
An important use case of SEs is to improve the efficiency of solving overdetermined linear least squares problems.
For a review on that matter, readers may consult \cite{woodruff2014sketching,martinsson2020randomized}.

Recent developments in the subject have concerned sketching nonlinear infinite subsets $V \sset \R^N$, so that for every $\x \in V$ it holds $\|\S\x\| \approx \|\x\|$.
We can preserve pairwise distances by solving this problem where $V$ is replaced by $V - V$ (the Minkowski sum $V + (-V)$); then  for all $\x, \y \in V$ we have $\| \S(\x - \y) \| \approx \| \x - \y \|$ (equivalently, $\| \S\x - \S \y \| \approx \| \x - \y \|$ if $\S$ is linear). 
Standard tools to control the sufficient sketching dimension are the covering number and the Gaussian width (see \cite{vershynin2018high}).
The covering number $\cN(V, \eps)$ is the smallest number of balls of radius $\eps$ needed to cover the set $V$.
The Gaussian width can be upper-bounded through the Dudley integral once the covering numbers are bounded.
If $\S$ is a sub-Gaussian sketch or an FJLT sketch, an upper bound on the Gaussian width implies an upper bound on the sketching dimension \cite{vershynin2018high,oymak2018isometric}.
Using this approach, \cite{yap2013stable,iwen2022fast} developed bounds for sketching dimensions when embedding smooth manifolds with or without boundaries.
In \cite{zhang2023covering}, the authors of the present paper proved a covering number bound for a class of sets including (compact subsets of) the image of polynomial maps, the image of rational maps, algebraic varieties (i.e., zero sets of polynomial equations), and semialgebraic sets (i.e., feasible sets of polynomial inequalities).  Thus \cite{zhang2023covering} provides sketching dimension bounds for all such subsets when using sub-Gaussian or FJLT sketch operators.

\subsection{Aim of this paper}\label{subsec:aim}

The goal of this paper is to provide new tools and theory for sketching images of polynomial maps and algebraic varieties using \textit{any} sketch operators.  In particular, we target structured sketch operators with computational advantages, such as tensor structured or sparse sketches.

To illustrate the idea, consider the following model problem. 
Let $\t T\in\R^{n^d}$ be a fixed tensor, with $d$ modes and having length $n$ in each mode. 
Consider the task of evaluating the Frobenius distance between $\t T$ and a low canonical polyadic (CP) rank tensor:
\begin{equation}\label{eq:cp}
    \t M = \sum\nolimits_{i = 1}^r \a_1^{(i)} \otimes \cdots\otimes \a_d^{(i)}.
\end{equation}
For convenience, denote this by $\t M = \cp(\A_1,\ldots,\A_d)$,
where $\A_j = (\a_j^{(1)} | \ldots |\a_j^{(r)}) \in\R^{n\times r}$. 
Evaluating $\|\t M - \t T\|$ directly requires $\cO(rn^d)$ flops, which is often too expensive in practice, especially when such distance is requested repeatedly for many choices of $\t M$, say in an optimization algorithm seeking for the optimal low rank approximation of $\t T$.

To overcome this, one could use a (linear) dimension reduction map $\S: \R^{n^d} \to \R^m$, $m\ll n^d$, so that for all tensors $\t M$ of all rank at most $r$ it holds
\begin{equation} \label{eq:example-model}
    \|\S\t M - \S\t T\| \approx \|\t M - \t T\|.
\end{equation}
This is equivalent to saying $\|\S\x\| \approx \|\x\|$ for all $\x$ in the image of the polynomial map $p: (\mathbb{R}^{n \times r})^d \rightarrow \mathbb{R}^{n^d}, (\A_1,\ldots,\A_d) \mapsto \t T - \cp(\A_1,\ldots,\A_d)$. 
To approximate the distance, we compute $\S\t T \in\R^m$ \textit{only once}.  Then for each input $\t M$, we evaluate $\S\t M$ and compute distances in $\R^m$.

For this model problem, the quality of the randomized dimension reduction depends on: (i) given an error tolerance, how large $m$ has to be so that the difference between $\|\S\x\|$ and $\|\x\|$ is within the tolerance with high probability; 
and (ii) how efficient it is to compute $\S\t M$ for $\t M$ given in a factorized form \eqref{eq:cp}. 
These two properties, corresponding to \textit{compression power} and \textit{efficiency} respectively, are important for other applications of randomized sketching as well.

In general, denote $V \subseteq \mathbb{R}^N$ the set to be sketched.
When $\S$ is a sub-Gaussian sketch or an FJLT sketch, \cite{zhang2023covering} proved a sketching dimension bound for $V$ being a polynomial image or a variety through bounding the covering number.
In particular, if $V$ is the image of $\mathbb{R}^n$  under a polynomial map with coordinate functions of degree at most $d$, 
the sketching dimension for a sub-Gaussian sketch is 
\begin{equation*}
    m = C\eps^{-2}(nd\log n + \log(1/\delta)),
\end{equation*}
where $\eps$ is the relative sketch error ($\|\S\x\|^2 \asymp (1\pm\eps)\|\x\|^2$) and $\delta$ is the failure probability.
This bound is close to optimal, in that for linear subspaces $V$ of dimension $n$, the sketching dimension bound is $C\eps^{-2}(n + \log (1/\delta))$.  However, the application of an unstructured sub-Gaussian matrix or an FJLT matrix is in general not efficient; for example, it is expensive when $V$ is the set of translate low rank CP tensors as in the model problem.  

On the other hand, there is a large volume of research on the theory {and applications} of structured sketches, including \it{tensor structured} or \it{sparse} sketch operators (e.g., \cite{jin2021faster,malik2020guarantees,ahle2020oblivious,haselby2023modewise,ma2022cost,hur2023generative,pham2013fast,bujanovic2024subspace}), such as OSNAP, the Kronecker FJLT (KFJLT), tensor sketch, Khatri-Rao sketch, and more generally tensor network structured sketches.
Many of these well designed sketch operators can be applied 
in near linear time (e.g., the KFJLT is applied in $\wt\cO(n)dr$ time to $\t M$ in the model problem).
Even for embedding general sets $V$ (e.g., where the tensor structure in $\t M$ is absent), tensor structured and sparse sketches always save storage and require fewer of random bits, and they are often faster to apply compared to sketches without a tensor structure.
However, to the best of our knowledge there is no available theory to control the sketching dimension $m$ for embedding the entire set CP tensors of rank $r$ or a general variety or a polynomial image.

To retain the advantages from both worlds, in this paper we provide new tools for deriving sketching dimension bounds for embedding varieties and polynomial images using \it{any sketching operator} $\S$, as long as the concentration behavior for $\S$ applied to any single vector is controlled.
The key is to connect sketching theory with \it{norming sets} (see \cref{def:norming} and, e.g., \cite{bos2018fekete} for more details).  Norming sets are a topic mainly from the field of approximation theory, but were recently linked to algebraic varieties using Hilbert functions of varieties in \cite{altschuler2023kernel}. 
Besides the work \cite{altschuler2023kernel}, the idea of norming set has also appeared implicitly in, e.g., \cite{siegel2022sharp,ergur2019probabilistic}.

\subsection{Our contributions} \label{sec:contrib}
In this paper, we establish that norming set theory is a powerful tool for bounding sketching dimensions, despite the fact that norming sets have not received their due attention.
Specifically, we 
provide a unified framework (\cref{thm:main}) for computing sufficient sketching dimension bounds when embedding a (subset of a) variety or polynomial image using any sketch operator $\S$, provided that the concentration behavior of $\S$ applied to a single vector is controlled.
\Cref{thm:main} extends the results in \cite{zhang2023covering} about sketching such 
sets to a much broader class of sketch operators, including tensor structured and sparse sketch operators. 
The new framework also implies that  (\cref{exp:subG}) if a sub-Gaussian sketch is used to sketch a polynomial image of $n$ variables under polynomials of degree at most $d$, then the sketching dimension is $\cO(n\log(nd))$, an improvement from the result $\cO(nd\log n)$ in \cite{zhang2023covering}.

As another main contribution, we propose a new sketching approach called median sketch, which is inspired by the median of means estimator in robust statistics.
By applying norming set theory, we prove that
on a variety or polynomial image $V\sset \R^N$,
median sketch achieves an accuracy of $\eps$ with failure probability at most $\delta$ using only  
\begin{equation*}  
    C \eps^{-2} \big(\dim(V) \log(N/\eps) + \log(1/\delta)\big)
\end{equation*}
tensor structured or sparse linear measurements (or any other type of sketch), 
provided that $V$ satisfies mild assumptions on its degree, and that on every point $\x\in V$, the sketch successfully embeds $\x$ to an accuracy $\eps$ with probability at least (say) $2/3$ using only $\cO(\eps^{-2})$ measurements. 
More details can be found in \cref{thm:medsketch}.
Our bound on the amount of linear measurements is near optimal, in that using a Gaussian sketch to embed a linear space $V$ already requires $C\eps^{-2}(\dim(V) + \log(1/\delta))$ measurements.  However, it allows for the use of structured linear measurements with desirable computational and/or storage costs.

\subsection{Notation} \label{sec:notation}

We use bold upper case calligraphic letters ($\t T,\t A, \ldots$) for tensors (usually of order at least 3), bold upper case letters ($\m T, \m A, \ldots)$ for matrices, and bold lowercase letters ($\v a, \v u, \ldots$) for vectors. 
The norm $\|\cdot\|$ by default refers to the $\ell_2$ (i.e., Frobenius) norm.
For two real numbers $a$ and $b$, we let $a \asymp (1\pm\eps) b$ indicate  $(1-\eps)b \leq a \leq(1 + \eps)b$, which is convenient notation for sketching theory. 
Denote the normalization map $\x \mapsto \x/\|\x\|$ for $\x \neq 0$ by $\pio$.
For a set $U \sset \R^n$, we denote $\pio(U)$ the projection of $U\setminus\set{0}$ to the unit sphere $S^{n-1}$.
We use $\overline{U}$ for the Euclidean closure of $U$, and $\zcl{U}$ for the closure with respect to the real Zariski topology \cite{10.5555/1197095,bochnak2013real}.
The function $\log(\cdot)$ denotes the base-$e$ natural logarithm.
We write $\im(p)$ for the image of a map $p$.

Throughout the paper, $C$ is reserved for representing positive universal constants.
Its value may change within an expression, as we think of it as a placeholder for a fixed number.
For instance, we may write $C f(x) + C g(x) \leq C h(x)$ to mean there are positive constants $\ga,\gb,\gc$ such that for all $x$ it holds $\ga f(x) + \gb g(x) \leq \gc h(x)$.
We use subscripted constants $C_1,~C_2$ when necessary or helpful, and 
these subscripted constants remain the same within a proof or context.

\section{From Norming Sets to Sketching Theory}\label{sec:norming}

We start by specifying the definition of a Johnson-Lindenstrauss transform used in the rest of the paper, as there are different versions in the literature.

\begin{definition}
    [JLT]
    A random map $\S : \mathbb{R}^N \rightarrow \mathbb{R}^m$ is an $(\eps, \delta)$ {\normalfont Johnson-Lindenstrauss transform (JLT)} on a set $U \subseteq \mathbb{R}^N$ if with probability at least $1-\delta$ it holds $\|\S\x\|^2 \asymp(1\pm\eps)\|\x\|^2$ for all $\x\in U$.
\end{definition}

Note we do not require preserving the pairwise distances in $U$.  
A JLT on $U-U$ will preserve the pairwise distances in $U$.
Next, we define the notion of norming set. 
Discussion on its background can be found in, e.g., \cite{bos2018fekete}.

\begin{definition}
    [norming set] \label{def:norming}
    Given a bounded set $V\sset \R^N$, a positive integer $d$, and a real number $\omega > 1$, we say that
    a subset $Q \subseteq V$ is a $(d,\omega)$ {\normalfont norming set} of $V$ if for all polynomials $p: \R^N \rightarrow \R$ of degree at most $d$ it holds 
    \begin{equation*}
        \|p\|_V \leq \omega \|p\|_Q,
    \end{equation*} 
    where the norm of $p$ over a bounded set is defined as $\|p\|_U = \sup_{\x \in U}|p(\x)|$.
    The collection of $(d, \omega)$ norming sets of $V$ is denoted by $\cM(V, d, \omega)$.  
\end{definition}

In fact, every compact set $V$ admits a finite norming set for any $d$ and $\omega$.\footnote{To see this, apply \cref{thm:norming} to $V_0 = \R^N$\!.}
A result in \cite{ergur2019probabilistic} showed that for 
the unit ball in $\R^N$, an $\eps$ net is a $(d, \,1/(1-\eps d^2))$ norming set.
To our knowledge, it is unclear if an $\eps$ net of any compact set $V$ is a norming set for $V$, and if so how big the parameter $\omega$ should be.
For our purposes a crucial result from \cite{altschuler2023kernel} bounds the size of a $(d, \omega)$ norming set of compact subsets of equi-dimensional varieties, which we state next.

\begin{restatable}[norming set of equi-dimensional varieties \cite{altschuler2023kernel}]{theorem}{normvar}
    \label{thm:norming}
    Let $V$ be a compact\footnote{This is with respect to the Euclidean topology.} subset of an equi-dimensional variety $V_0$ in $\R^N$.
    Let $D$ and $n$ be the degree and dimension of $V_0$.
    Then for any positive integer $d$ and real number $\go > 1$, there exists a norming set $Q\in\cM(V, d, \omega)$ with cardinality satisfying
    \begin{equation}\label{eq:pablo}
        {\log |Q| \leq C_1\log D + C_1n(\log(C_2nd) - \log\log \omega)}.
    \end{equation}
\end{restatable}

If $n = 0$, then the bound \eqref{eq:pablo} is still valid, interpreting $n\log(nd)$ as the limit $0$. 
The proof of \Cref{thm:norming} relies on a bound for the Hilbert function of (the projectivization of) $V_0$ which counts the number of linearly homogeneous polynomial functions on $V_0$ of each degree; see \cite{altschuler2023kernel}.

To establish the connection between norming sets and sketching theory using a random linear operator $\S$, consider the sketching task
\begin{equation}\label{eq:sketchgoal}
    \|\S\x\|^2 \asymp (1\pm\eps) \|\x\|^2
\end{equation}
for all $\x \in V$.
This is equivalent to the following condition on the norm of a degree $4$  polynomial:
\begin{equation}
    \|p\|_{\pio(V)} \leq \eps^2,~~\text{where}~~ p(\x) = \left(\|\S\x\|^2 - 1\right)^2\!\!.
\end{equation} 
Since $\pio(V)$ is a precompact subset of the variety $\zcl{\pio(V_0)}$ (the closure with respect to the Zariski topology, see \cref{sec:notation}), 
we can now use 
\Cref{thm:norming} 
to reduce \eqref{eq:sketchgoal} to a condition on a finite set $Q$.
A union bound can then be used to further reduce the condition to a single point in $\overline{\pio(V_0)}$. 
This allows us to derive a sketching dimension bound for the entire set $V$ by only analyzing the concentration behavior of the sketch operator on a single vector.
This result is summarized below in \cref{thm:main} and \Cref{cor:reducibleV}.

To better formulate our main result, we define a few notions for sketch operators.

\begin{definition}
    [ensemble of sketch operators]
    An {\normalfont ensemble of sketch operators} on $\R^N$ is a sequence of laws of random matrices $E = \set{E_m}_{m\in\Z_+}$, where $E_m$ is the law of an $m\times N$ random matrix.
\end{definition}

As an example, the Gaussian sketch ensemble is given by $E_m \sim m^{-1/2} \G_m$, where $\G_m$ has $m$ rows populated with i.i.d. standard Gaussian entries.
The next definition characterizes the concentration behavior of a sketching ensemble on a single vector.
\begin{definition}[exponential restricted isometry property]
    For $U\sset\R^N$, we say an ensemble $E$ on $\R^N$ has the {\normalfont exponential restricted isometry property (eRIP)} on $U$ with tail function $\phi$, if for any $m\in\Z_+$, $\eps\in (0, 1)$, and any fixed $\x\in U$,
    \begin{equation*}
        \pr_{\S\sim E_m}\left(\|\S\x\|^2\asymp(1\pm\eps)\|\x\|^2\right) \geq 1 - e^{-\phi(m, \eps)}.
    \end{equation*}
    We denote this by $E\sim \erip(U, \phi)$.
\end{definition}

\begin{remark}
    Note that for some sketch operators, the tail bound function $\phi$ may depend on e.g., the ambient dimension $N$. We exclude such potential dependencies in the notation, viewing them as fixed constants in the $\phi$ function.
\end{remark}

Note that the larger the $\phi$ function, the stronger the concentration.
Many commonly used sketch operators have $\erip$ on all of $\R^N$. 
For example, it is well known that the Gaussian sketch ensemble has eRIP on $\R^N$ with $\phi(m, \eps) = Cm\eps^{2}$.
More generally, if a bound on 
\begin{equation*}
    {\sup_{\x\in \overline{\pio(U)}}}\E_{\S\sim E_m} \big|\|\S\x\|_2^2 - 1\big|^p
\end{equation*}
is available, e.g., when $\S$ satisfies the so called (strong) JL moment property (see, e.g., \cite{woodruff2014sketching}), then a $\phi$ function can be derived for $\S$ using Markov's inequality.
We give concrete examples of this in \cref{sec:past}.

Our main result on sketching varieties and polynomial images based on norming set theory can now be stated.  It is stated for irreducible varieties and polynomial images.  An easy generalization with reducible varieties (i.e., finite unions of irreducible varieties) is in \Cref{cor:reducibleV}.

\begin{theorem}[restate = main, name = sketching irreducible varieties and polynomial images]\label{thm:main}
    Suppose we sketch set $V\sset \R^N$ with sketching ensemble $E$ such that $E\sim\erip\Big(\overline{\pio(V)}, \phi\Big)$.
    If $V$ is a subset of an irreducible variety $V_0 \sset \R^N$ of degree $D$ and dimension $n$,
    then for any $\eps, \delta\in (0, 1)$, $\S \sim E_m$ is an $(\eps, \delta)$ JLT on $V$ 
    provided that
    \begin{equation}\label{eq:Vbnd}
        \phi\left(m, \frac{\eps}{\sqrt{2}}\right) \geq {C_1\log D + C_1 n \log(C_2n) + \log(1/\gd)}.
    \end{equation}
    If instead $V$ is a subset of the image of polynomial map from $\R^n$ to $\R^N$ whose coordinate functions each have a degree of at most $d$,
    then the above assertion remains true with $\log D$ in \eqref{eq:Vbnd} replaced by $n\log d$.
\end{theorem}

\begin{proof}
    [Proof Outline]
    The full proof involves technicalities in algebraic geometry, and can be found in \cref{sec:pfmain}.

    Define $p(\x) = (\|\S\x\|^2 - 1)^2$,  a degree $4$ polynomial. 
    Since $\overline{\pio(V)} \sset \zcl{\pio(V_0)}$, using \cref{thm:norming}, we can prove there is a norming set $Q\in \cM\Big(\overline{\pio(V)}, 4, 2\Big)$ of cardinality 
    \begin{equation*}
        \log|Q| \leq {C_1\log D + C_1n\log(C_2n)}.
    \end{equation*} 
    Applying a union bound to the points in $Q$, we have $\|p\|_Q \leq \eps^2/2$ with probability at least $1 - |Q|e^{-\phi(m, \eps/\sqrt{2})}$.
    Consequently, with at least the same probability, $\|p\|_{\pio(V)} \leq \eps^2$.
    In order for the probability lower bound to be no less than $1 - \delta$,
    it suffices to have
    \begin{equation*}
        \phi(m, \eps/\sqrt{2}) \geq \log |Q| + \log(1/\delta).
    \end{equation*}
    This condition is precisely \eqref{eq:Vbnd}.
    In the polynomial map case, by B{\'e}zout's theorem the degree is bounded by $\log D = n\log d$.
    The proof is thus complete.
\end{proof}

\begin{remark}\label{rmk:skip_union}
    In the argument for \cref{thm:main}, a union bound is used to ensure $\|\S\x\|^2 \asymp (1\pm\eps)\|\x\|^2$ for all $\x\in Q$.
    Essentially, we only need $\S$ to be a JLT on the finite set $Q$.
    Thus, in \cref{thm:main} the eRIP condition on the ensemble could be replaced by the condition that $\S$ is a JLT on a finite (unspecified) set $Q$. 
    We choose the eRIP formulation because it is often easier to verify, and in many cases the bound given by the eRIP condition plus a union bound is identical to the bound derived by requiring $\S$ to be a JLT on $Q$.
\end{remark}

To illustrate how to apply \cref{thm:main}, let us 
derive specific sketching dimension bounds for two types of frequently used sketch operators, the sub-Gaussian ensemble and FJLT \nolinebreak ensemble.

\begin{example}
    [sub-Gaussian ensemble]\label{exp:subG}
    Let $E_m \sim m^{-1/2} \G_m$, where $\G_m$ is a random matrix with $m$ rows whose entries are independent, mean zero, unit variance, sub-Gaussian random variables.
    Let the $\psi_2$ norm of the entries be bounded by $K \geq 1$.
    Then Bernstein's inequality implies that for $\eps < 1$ it is valid to take
    \begin{equation*}
        \phi(m, \eps) = C m \cdot \min(\eps/K, \eps^2/K^2) = C m \eps^2/K^2.
    \end{equation*}
    Thus, in order to sketch a subset $V$ of a variety in \cref{thm:main} with probability at least $1-\delta$, the sufficient sketching dimension is
    \begin{equation*}
        m_v = C_1\eps^{-2} K^2 \cdot \left(\log D + {n\log(Cn)} + \log(1/\delta)\right),
    \end{equation*}
    For a polynomial map $p$ from $\R^n$ to $\R^N$ whose coordinate maps have degree at most $d$, the sufficient sketching dimension is
    \begin{equation*}
        m_p = C_2\eps^{-2}K^2 \cdot ({n\log(Cnd)} + \log(1/\delta)).
    \end{equation*} 
\end{example}

\begin{example}
    [FJLT ensemble] 
    Let $E$ be the FJLT ensemble on $\R^N$.
    This means that a realization $\S_m\sim E_m$ is given by $\S_m = \sqrt{\frac{N}{m}}\P_m\H\D$, where $\D$ is a diagonal matrix of Rademacher variables, $\H$ is the Walsh-Hadamard transform, and $\P_m$ is a uniform sampling of $m$ rows.
    A recent result in \cite[Corollary 2.5]{iwen2022fast} shows $\S_m$ is a JLT on $k$ vectors in $\R^N$ with probability at least $1-\delta$ provided  
    \begin{equation*}
        m \geq C \eps^{-2} \log(k/\delta)\cdot \Big[\log^2\left(\eps^{-1} \log(k/\delta)\right) \cdot \log N + \log(1/\delta)\Big]. 
    \end{equation*}
    Thus, as pointed out in \cref{rmk:skip_union}, we can use this result directly. 
    Consequently, to embed the subset $V$ using FJLT with probability at least $1-\delta$, the sufficient sketching dimension is
    \begin{equation*}
        m_v = C_1\eps^{-2} \Delta_v \cdot \left[\log^2(\eps^{-1}\Delta_v)\log N + \log(1/\delta)\right],
    \end{equation*}
    where $\Delta_v = {C\log D + C n \log(Cn)} + \log(1/\delta)$.
    For a polynomial map $p:\R^n\rightarrow \R^N$ with coordinate functions of degree at most $d$, the sufficient sketching dimension for $\im(p)$ is 
    \begin{equation*}
        m_p = C_2\eps^{-2} \Delta_p \cdot \left[\log^2(\eps^{-1}\Delta_p)\log N + \log(1/\delta)\right],
    \end{equation*}
    where $\Delta_p = {C n \log(Cnd)} + \log(1/\delta)$.
\end{example}

Thus for sub-Gaussian and FJLT sketches, compared to  \cite{zhang2023covering} which derives the sketching dimension using the covering number bounds for $\im(p)$,
the new bound here gives a slight improvement.  It updates the sketching dimension from $\wt\cO(nd\log(n))$ to $\wt\cO(n\log(nd))$.

{\cref{cor:reducibleV} below extends} \cref{thm:main} to the case of reducible varieties.  
We simply take a union of norming sets corresponding to each irreducible component to get a norming set corresponding to the whole variety. 
As a particular case of \cref{cor:reducibleV}, if all irreducible components of the variety $V_0$ have the same dimension so that $V_0$ is equi-dimensional, then we recover the bound \eqref{eq:Vbnd} in \cref{thm:main}.

\begin{corollary} 
    [sketching reducible varieties]\label{cor:reducibleV}
    Let $V_0\sset \R^N$ be a variety, and $V\sset V_0$ be the set to be sketched.
    Suppose ensemble $E\sim \erip(\overline{\pio(V)}, \phi)$ is used to sketch $V$.
    Let $\set{V_i}_i$ be the finite collection of irreducible components of $V$, with degree $D_i$ and dimension $n_i$ respectively.
    Then for any $\eps, \delta\in (0, 1)$, $\S \sim E_m$ is an $(\eps, \delta)$ JLT on $V$ 
    provided that
    \begin{equation}\label{eq:redVbnd}
        \phi\left(m, \frac{\eps}{\sqrt{2}}\right) \geq \log \sum\nolimits_i {\left({D_i}^{C_1} \cdot {(C_2n_i)}^{C_1n_i}\right)} + \log(1/\gd).
    \end{equation}
\end{corollary}

\begin{proof}
    Following the argument of \cref{thm:main}, we find norming sets $Q_i\in \cM\Big(\overline{\pio(V_i \inter V)}, 4, 2\Big)$ of size
    $|Q_i| \leq {{D_i}^{C_1} \cdot {(C_2n_i)}^{C_1n_i}}$.
    Then $Q = \Union_i Q_i \in \cM\Big(\overline{\pio(V)}, 4, 2\Big)$ with cardinality
    \begin{equation*}
        \log |Q| \leq \log \sum\nolimits_i {\left({D_i}^{C_1} \cdot {(C_2n_i)}^{C_1n_i}\right)}.
    \end{equation*}
    Repeating the union bound in the proof outline for \cref{thm:main} we obtain \eqref{eq:redVbnd}.
\end{proof}

Thanks to the norming set theory, \cref{thm:main} and \cref{cor:reducibleV} are quite flexible in that they are applicable to virtually any sketch operator, including tensor structured and sparse sketches. 
In comparison, developing sketching theory for such sketch operators using covering numbers often requires substantial work \cite{oymak2018isometric,li2017near}, tailored to the given type of operator at hand.

\section{Sketching Polynomial Images and Varieties by Tensor Structured Sketches}\label{sec:median}

In this section, we demonstrate the 
broad applicability of \cref{thm:main}  to structured sketch operators.  Let $V$ be a subset of an equi-dimensional variety or the image of a polynomial map that we aim to sketch.
In many applications, the set $V$ itself has 
special structure.  For example, consider sparsity structure in the applications of compressed sensing, or tensor based structure in tensor decomposition, polynomial kernel approximation, moment based model fitting, and certain neural networks (see e.g., \cite{sherman2020estimating,altschuler2023kernel,kolda2009tensor,ahle2020oblivious,zhang2023moment,jin2024scalable,pereira2022tensor,kacham2023polysketchformer}). 
In such cases, it would be desirable computationally to use compatibly structured sketch operators.
Here we focus on the case of tensor structure, recalling the model problem in \cref{subsec:aim}.
To embed tensor structured $V$, we would like to use a tensor structured sketch operator that can be applied to $\x \in V$ efficiently.
For example, if $V \sset \R^{n^d}$ is the set of low CP rank tensors, and $\x$ is given in a factorized form $\cp(\A_1,\ldots,\A_d)$, then applying a tensor product sketch $\S = \bigotimes_{i = 1}^d\S_i$ to $\x$ only takes $\cO(n)$ flops and storage, whereas a general Gaussian sketch matrix of the same output dimension would require $\cO(n^d)$ flops and storage.
Other common choices of tensor structured sketch operators include KFJLT \cite{malik2020guarantees,jin2021faster}, tensor sketch (TS) \cite{pagh2013compressed}, Khatri-Rao (KR) sketch \cite{ahle2020oblivious}, etc.
Theory for sketch operators that take a general tensor network structure is established in \cite{ahle2020oblivious}. See also \cite{ma2022cost} for follow up developments and efficient applications of such network sketches.

\subsection{Challenge and previous attempts}\label{sec:past}

Recall the requirements (i) and (ii) for a good sketch operator in \cref{subsec:aim}.
Despite their efficiency, the challenge with tensor structured sketch operator is their accuracy guarantees are typically weak.  
For a tensor structured sketch operator of tensor order $d$ (say  KFJLT or a KR sketch), best known sketching dimension bounds for these operators have a $\phi$ function of order
\begin{equation}\label{eq:phid}
    \phi(m, \eps) = \cO(m^{1/d}),
\end{equation}
where we treat everything but $m$ as constants.
See \cite{ahle2020oblivious,bamberger2022johnson}.
Equivalently, for such operators to be a JLT over $P$ points, the sketching dimension must to be $\Omega(\log^d P)$.
The next example provides a justification for \eqref{eq:phid} through the lens of the (strong) JL moment property.

\begin{example}
    [moment bound implies eRIP]\label{exp:momerip}
    Fix $U\sset \R^N$ and the ensemble $E = \set{E_m}_m$. 
    Suppose for all $m$, the random matrix $\S\sim E_m$ satisfies the following moment bound.
    There exists $d \geq 1$ such that for all\footnote{Actually, this is not necessary. 
    We only require the moment bound to hold at the optimal $p$ minimizing the right-hand side of  Markov's inequality.} $p \geq 1$ it holds
    \begin{equation}\label{eq:mombnd}
        \sup_{\x\in \overline{\pio(U)}}\left(\mathbb{E}_{\S \sim E_m} \left|\|\S\x\|_2^2 - 1\right|^p\right)^{1/p}  \, \leq \, C_d \left(\frac{p^d}{m} + \sqrt{\frac{p}{m}}\right),
    \end{equation} 
    where $C_d$ is some constant potentially depend on $d$.
    The bound \eqref{eq:mombnd} is common for tensor structured sketch $\S$.
    For example, if the rows in $\S$ are independent sub-Weibull measurements,
    e.g., if $\S$ is a Khatri-Rao sketch, where the rows are independent and have the form $(\a_1\otimes\cdots\otimes \a_d)^\top$ with each $\a_i$ is a sub-Gaussian vector, then \eqref{eq:mombnd} applies \cite{latala1997estimation,ahle2020oblivious,bong2023tight}.

    From \eqref{eq:mombnd}, in order to obtain 
    \begin{equation*}
        \pr\left(\|\S\x\|^2 \asymp (1\pm\eps)\|\x\|^2\right) \geq 1 - \delta,
    \end{equation*}
    one applies the Markov's inequality and optimizes over $p$ to arrive at a sufficient sketching dimension
    \begin{equation*}
        m = \max\set{C_{1,d} \eps^{-1}\log^d(1/\delta),~C_{2,d} \eps^{-2}\log(1/\delta)}
    \end{equation*}
    for some constants $C_{1,d}$ and $C_{2,d}$ (see, e.g., the computation in the appendix of \cite{ahle2020oblivious}). 
    Replacing $\delta$ with $e^{-\phi}$, we deduce a sufficient $\phi$ function of order
    $\phi(m, \eps) = \cO(m^{1/d})$.
\end{example}

Though \cref{exp:momerip} only derives sufficient conditions, recent results also show the bound $\Omega(\log^d P)$ is necessary for KFJLT operators of order $d$ to embed $P$ points \cite{bamberger2022johnson}. So, the $m^{1/d}$ behavior in \eqref{eq:phid} is much expected.

Let us now explain why the order of $m^{1/d}$ is very bad.  
Consider the model problem of sketching the set of all rank-1 tensors of shape $n^d$.
Since the dimension of this set is $b = \Theta(nd)$, for tensor structured sketch operators with a $\phi$ function as in \eqref{eq:phid}, the sufficient sketching dimension is 
\begin{equation} \label{eq:embedtensor}
    \phi(m, \eps/\sqrt{2}) \geq C b \log(b d),
\end{equation}
which implies $m$ is at least $\Omega(b^d) = \Omega(n^d)$.
In other words, there is no compression at all!

Some attempts have been made to address the $m^{1/d}$ scaling in the $\phi$ function, or equivalently, to improve the sufficient sketching dimension from $\log(1/\delta)^d$ for an order $d$ tensor structured sketch to have a failure rate at most $\delta$.
A noticeable stream of work on this is \cite{ahle2020oblivious,ma2022cost}.
The idea is to use tensor structured sketches of order only $2$, and use a binary tree formed by these order 2 sketches to embed $\R^{n^d}$ tensors to $\R^m$.
The paper \cite{ahle2020oblivious} proved that using a binary tree of order 2 tensor sketches, the dependence of the final sketching dimension $m$ on failure probability $\delta$ is given by the dependence of each order 2 sketch in the tree, thus reducing the sketching dimension to $m(\delta) = \cO\left(\log^C(1/\delta)\right)$.
Equivalently, $\phi(m) = \Omega(m^{1/C})$.
However, even with this hierarchical sketching scheme, the final output dimension has to be $m = \Omega(b^C) = \Omega(n^C)$ to satisfy \eqref{eq:embedtensor}. Since $C \geq 2$, which is the order of the sketch nodes in the tree, this sketching dimension is still too large to be considered optimal. 
Another way to seek a near optimal sketching dimension of $m = \wt\cO(b)$ (since the set to be sketched has a dimension $b$) would be to use a binary tree of $\R^{n^2} \to\R^m$ and $\R^{m^2}\to\R^m$ sketch operators without tensor structures.
However, if the operators $\S$ in the tree are sub-Gaussian sketches, this will incur a computational cost of $\Omega(m^3) = \Omega(n^3)$ at each tree node; if the operators are FJLT, to our knowledge the optimal known sketching dimension bound for FJLT has a dependence $m(\delta) = \Omega(\log^2(1/\delta))$ (see, e.g., \cite{cohen2015optimal,ahle2020oblivious}), and hence the final output dimension has to be $\Omega(b^2)$.\footnote{The result in \cite{iwen2022fast} shows the sketching dimension of FJLT for embedding $P$ points is $m = \wt\cO(\log(P/\delta))\cdot \log(1/\delta)$ using the RIP theory, which is better than $\cO(\log^2(P/\delta))$ from bounding the moments. But it is unclear if the RIP result can be extended to a tensor network of FJLT operators to ensure a $\wt\cO(\log(P))$ dependence on $P$.}
Hence, neither of these is perfect.

\subsection{Approach of median sketch}\label{sec:new}

To obtain tensor structured sketches that are both efficient and accurate, we propose a new approach called median sketch.  In fact, leveraging norming set theory we will establish bounds to sketch the set $V \subseteq \mathbb{R}^N$, a subset of a variety or image of polynomial map, using any sketch.

Intuition behind the median sketch is long established.
It is well known that if a random sketch succeeds with $\cO(1)$ chance, then independently repeating the sketch for $\cO(\log (1/\delta))$ times can improve the failure chance from $\cO(1)$ to $\delta$. 
As such, the median sketch approach proceeds as follows.  
We generate a committee of several sketch operators, $\S_1,\ldots,\S_{2k+1}$ independently from $E_m$.   We measure $a_i = \|\S_i\x\|$, and take $\S_{i^*}\x$ with median norm as our sketched output.

To be more precise, 
given a sequence of numbers $a_1,\ldots,a_\ell$, let
\begin{equation*}
    i^* = \argmed_i a_i
\end{equation*}
return the index corresponding to the median of the array.
If there is a tie among the numbers, take the smallest $i$ by convention.
We always guarantee that the size $\ell$ is odd, so that $i^*$ is well defined.
We thus denote the median by
\begin{equation*}
    a_{i^*} = \med_i a_i.
\end{equation*}
We now give full details of the proposed median sketch approach in \cref{alg:median}.

\begin{algorithm}[!h]
    
    \textbf{Input:} sketch operators $\S_1,\ldots,\S_{2k+1}$, vector $\x$\\
    \textbf{Output:} sketched vector $\y$ so that $\|\y\| \approx \|\x\|$
    
    \vspace*{1ex}
    \begin{algorithmic}[1]
        \caption{Sketch using the median of the committee}\label{alg:median}
        \Function{MedianSketch}{}
        \State compute $\y_i \gets \S_i \x$ and $a_i\gets \|\y_i\|$ for $i = 1,\ldots,2k+1$
        \State $i^* \gets \argmed_i a_i$
        \EndFunction
        \\
        \Return $\y_{i^*}$
    \end{algorithmic}
\end{algorithm}

A caveat of median sketch is that the sketching operation is no longer a linear map, but a piecewise linear map.  This turns out to be necessary to obtain a favorable sketching dimension (\cref{thm:medsketch}).
To remind ourselves of the nonlinearity, when the committee is understood from context, we denote the sketching operation as $\y = \tilde S(\x)$.
The map is still scaling homogeneous, i.e.,
for $a\in\R$ and $\x \in \R^N$, $\tilde S(a\x) = a\tilde S(\x)$.
In \cref{alg:medjlt}, we show how to estimate various pairwise distances in an infinite set $V$ using median sketch, in spite of the loss of linearity.

\begin{algorithm}[!h]
    
    \textbf{Input:} sketch operators $\S_1,\ldots,\S_{2k+1}$, data points $ \x_i \in  V$ ($i =1, \ldots, P$)\\
    \textbf{Output:} all sketched pairwise distances $\wh{d}_{ij} \approx d(\x_i, \x_j)$
    
    \vspace*{1ex}
    \begin{algorithmic}[1]
        \caption{Fast pairwise distance evaluation of datasets in $V$ using median sketch}\label{alg:medjlt}

        \Function{MedianJLT}{}
            \For{$i = 1,\ldots,P$}
            \Comment{sketch the dataset}
                \State compute the profiles $\y_{ij}\gets \S_j \x_i$ for all $j = 1,\ldots,2k+1$
            \EndFor
            \For{$i = 1,\ldots,P$}
                \For{$j = i+1,\ldots,P$}
                \Comment{compute the sketched pairwise distance $\wh d_{ij}$}
                    \State $a_s \gets \|\y_{is} - \y_{js}\|$, $s = 1,\ldots,2k+1$
                    \State $s^*\gets \argmed_s a_s$
                    \State $\wh d_{ij}\gets a_{s^*}$
                \EndFor
            \EndFor
        \EndFunction
        \\
        \Return distance array $\wh d_{ij}$
    \end{algorithmic}
\end{algorithm}

In \cref{alg:medjlt}, if the committee successfully preserves norms on the set $V-V$, i.e., for any $\u\in V-V$ it holds that $\|\tilde S(\u)\|^2 \asymp (1\pm\eps)\|\u\|^2$, then $\wh d_{ij} \asymp (1\pm\eps) \cdot d(\x_i,\x_j)$ for $\x_i, \x_j \in V$.
If the sketching dimension of each $\S_i$ is $m$ and the ambient dimension of $V$ is $N$, then up to the overhead of computing the sketch profile whose complexity is linear in $P$, the total time to compute all $\wh d_{ij}$ is $\cO(kmP^2)$.
Without sketch, this cost can be as large as $\cO(NP^2)$.  Therefore, the benefit of the median sketch depends on the size of $km$ (which is independent of $P$), i.e., the total number of linear measurement required by median sketch.
The next main result provides a bound on this quantity.

\begin{restatable}[sketching dimension bound of median sketch]{theorem}{medsketch}
    \label{thm:medsketch}
    Let $V_0$ be a variety of dimension $n_v$ and degree $D$.
    Let $V$ be the subset of $V_0$ to be sketched.
    Let the desired sketching error be $\eps\in (0, 1)$.
    Let $\S\sim E_m$ be a random matrix such that 
    \begin{enumerate}
        \item for any $\S$ in the support of $E_m$ and $\x\in {\pio(V)}$, $\|\S\x\| \leq M$ (see also \cref{rmk:Mbnd}); and
        \item for any fixed $\x \in \overline{\pio(V)}$ and some $\theta > \log 4$,
        \begin{equation}\label{eq:theta}
            \pr\left(\|\S\x\|^2 \geq 1+\frac{\eps}{2}\right) \leq e^{-\theta}~\text{and}~\pr\left(\|\S\x\|^2 \leq 1-\frac{\eps}{2}\right) \leq e^{-\theta}.
        \end{equation}
    \end{enumerate}
    Then for an i.i.d. committee $\S_1,\ldots,\S_{2k+1}\sim E_m$, the corresponding median sketch map $\tilde S$ satisfies
    \begin{multline*}
        \pr\big(\forall \x\in V,~\|\tilde S(\x)\|^2 \asymp (1\pm\eps)\|\x\|^2\big) \\
        \geq 
        1 - \exp\left[-C_1(\theta - \log 4)k + {C_2}\log D + C_3 n_v\log\left(\frac{n_vMk}{\eps}\right)\right].
    \end{multline*}
    If instead $V_0$ is the image of a polynomial map with coordinate functions of degree at most $d$, then the above inequality remains valid by setting $\log D =  {n_v\log d}$.

    In particular, for both the variety and the polynomial map cases, taking $M \geq n_v$, $\theta = 1+\log 4$, and assuming $\log D = \cO(n_v\log n_v)$, then for any $\delta \in (0, 1)$,
    \begin{equation*}
        \pr\big(\forall \x\in V,~\|\tilde S(\x)\|^2 \asymp (1\pm\eps)\|\x\|^2\big) \geq 1 - \delta
    \end{equation*}
    provided that $k \geq C_4 \left(n_v\log \left(M/\eps\right) + \log (1/\delta)\right)$.
\end{restatable}

\begin{remark}\label{rmk:Mbnd}
    The boundedness assumption $\|\S\x\|^2\leq M$ can be replaced by almost boundedness.
    That is, for any $\x\in {\pio(V)}$,
    \begin{equation*}
        \pr(\|\S\x\|^2\geq M) \leq e^{-\beta}
    \end{equation*}
    for some $\beta$.
    Then the arguments in the proof will carry through, and the failure probability will also depend on $\beta$ (see \cref{rmk:Mbnd-pf} for more details).
    That said, many commonly used sketching operators are indeed bounded.
    For example, if $\S$ is the KFJLT sketch or the Khatri-Rao sketch with Rademacher entries, then $M = \cO(N)$ and $\log M = \cO(\log N)$. 
\end{remark}

\begin{proof}
    [Proof Outline]
    The details of the proof of \cref{thm:medsketch} take a few steps; a complete proof is delayed to \cref{sec:pfmed}.
    By scaling homogeneity of $\tilde S$, it suffices to have
    \begin{equation*}
        (\|\tilde S(\x)\|^2 - 1)^2 \leq \eps^2~\text{for all }\x\in \overline{\pio(V)}.
    \end{equation*}
    Unlike in \cref{sec:norming}, the left-hand side is no longer corresponds to the norm of any polynomial on $V$ since $\tilde S$ is nonlinear.
    However, using the approximation power of polynomials on compact sets, which is essentially why we have condition (i), we can construct polynomials $P(\x)$ and $R(\x)$ to approximately count  $\text{\small \#}\set{i: \|\S_i\x\|^2 \geq 1+\eps}$ and $\text{\small \#}\set{i: \|\S_i\x\|^2 \leq 1-\eps}$ respectively.
    We upper bound the polynomials $P$ and $R$ on the entire set $V$ by bounding them on a norming set.
    Once $P$ and $R$ are bounded, the median is controlled.
\end{proof}

Fixing $\theta = \cO(1)$, the sketching dimension $m$ needed for \eqref{eq:theta} usually scales no worse than $\cO(\eps^{-2})$.
The assumption $\log D = \cO(n_v\log n_v)$ is satisfied, if in both the variety and the polynomial cases the defining polynomials of $V$ have a degree at most $\poly(n_v)$.
Combining these and \cref{rmk:Mbnd}, the total number of measurement is
\begin{equation*}
    mk \sim C \eps^{-2}(n_v \log(N/\eps) + \log(1/\delta)).
\end{equation*} 
This shows the bound on the total number of measurement is near optimal.
Even if the set $V$ is a linear subspace and unstructured Gaussian measurements are used, the total number of measurements is already $\Omega\big(\eps^{-2}(n_v + \log(1/\delta))\big)$ (see e.g., \cite{woodruff2014sketching}). 

If in particular $V$ is the set of all $n^d$ tensors of CP rank at most $r$, then $n_v \leq ndr$, and
\begin{equation*}
    mk = C\eps^{-2}\big[ndr \cdot (d\log n + \log(1/\eps)) + \log(1/\delta)\big].
\end{equation*}
This agrees with the sketching dimension bound given in \cite{zhang2023covering} when Gaussian sketch operators are used, up to the constant and the $\log(1/\eps)$ term.
Yet, the new result here applies to general sketch operators.
A possible application of \cref{alg:medjlt} is therefore to the model problem of low CP rank approximation (\cref{subsec:aim}).  Given a tensor $\t T$, median sketch with tensor structured sketches could be used to efficiently evaluate the loss $\|\t T - \cp(\A_1,\ldots,\A_d)\|^2$ for any $\A_1,\ldots,\A_d$.

\section{Proofs}\label{sec:proof}

\subsection{Proof of \cref{thm:main}}\label{sec:pfmain}

\begin{proof}
  Define $p(\x) = (\|\S\x\|^2 - 1)^2$. 
It suffices to show that 
\begin{equation*}
    \|p\|_{\overline{\pio(V)}} \leq \eps^2
\end{equation*}
holds with probability at least $1 - \delta$.

View $\overline{\pio(V)}$ as a compact subset of the variety $\zcl{\pio(V_0)}$.   
We will argue that $\zcl{\pio(V_0)}$ is equi-dimensional with dimension at most $n$ and degree at most $2D$.    
Firstly if $V_0$ is a cone (i.e., closed under scalar multiplication), $\zcl{\pio(V_0)} = V_0 \cap S^{N-1}$.  Then $\zcl{\pio(V_0)}$ is equi-dimensional with dimension $n-1$, and has degree at most $2D$ by B{\'e}zout's theorem. 
Else $V_0$ is not a cone, and we regard $\zcl{\pio(V_0)}$ as the Zariski closure of the image of the variety $W = \{(\x, \lambda) \in \mathbb{R}^{N} \times \mathbb{R} : \x \in V_0, \lambda^2 = \| \x \|^2 \}$ under the rational map $\psi(\x,\lambda) = \x / \lambda$. 
Here $W$ is equi-dimensional with dimension $n$, and has degree at most $2D$.  
By assumption $\psi|_{W}$ is generically finite-to-$1$, whence $\zcl{\pio(V_0)}$ is equi-dimensional with dimension $n_v$.  
Further, the degree of $\zcl{\pio(V_0)}$ is at most $2D$;  indeed, pull back the intersection of (the complexification of) $\zcl{\pio(V_0)}$ with a generic affine subspace of codimension $n$ via $\psi$ to (the complexification of) $W$ and use that the coordinate functions of $\psi$ are quotients of degree $1$ functions. 
This shows the dimension and degree bound.

    Next we apply \cref{thm:norming} to $\overline{\pio(V)} \subseteq \zcl{\pio(V_0)}$.  
    It gives a norming set $Q \in \cM(\zcl{\pio(V)}, 4, 2)$ with cardinality
    \begin{equation*}
        \log |Q| \leq C_1\log D + C_1 n\log(C_2n).
    \end{equation*}
In the case where $V_0$ is the Zariski closure of a polynomial image, then we can bound the degree $D$ by $d^n$. The rest of the proof can be completed following the proof outline in \cref{sec:norming}. 
\end{proof}

\subsection{Proof of \cref{thm:medsketch}} \label{sec:pfmed}

The condition we need is $\med_i \|\S_i \x\|^2 \asymp 1\pm \eps$ for all $\x\in\pio(V)$.
Let us focus on one side $\med_i \|\S_i \x\|^2 \leq 1 + \eps$, and the other side will follow from an anolgous argument.
To check this one-sided condition, we define the counting function
\begin{equation*}
    \wb p(x) = \ind{x > 1 + \eps}.
\end{equation*}
Then we require $\sum_{i = 1}^{2k+1} \wb p(\|\S_i\x\|^2) \leq k$ for all $\x \in \pio(V)$.
If we are able to approximate $\wb p$ by a nonnegative polynomial $p$, then we are left to check if $\|p\|_{\pio(V)} \leq k$.
This can be done using the norming set. 
We do this by first approximating the indicator function $\wb p$ with a continuous piecewise linear map, and then approximate the piecewise linear map with a polynomial using a classical result by Bernstein stated below.

\begin{lemma}[Bernstein]\label{lem:relu}
    For all $d\in\Z_+$, there exists a polynomial $p$ of degree at most $d$ such that on $[-1,1]$ it holds that $|p(x) - \relu(x)| \leq Cd^{-1}$, where $\relu(x) = x\ind{x\geq 0}$.
\end{lemma}

Applying Bernstein's result, we give an approximation result for the counting function below.

\begin{lemma}\label{lem:approxcnt}
    For any $\eps \in (0, 1)$, $M \geq 3$, and $\eta \in (0, 1/2)$, there is a polynomial $p$ of degree $d \leq \frac{CM}{\eps\eta}$ on $[0, M]$ such that (i) $p(x) \in [0, 1]$; (ii) $x\in[0, 1+\eps/2]$ implies $p(x) \leq \eta$; and (iii) $x\in[1+\eps, M]$ implies $p(x) \geq 1- \eta$.
\end{lemma}

\begin{proof}
 Consider the piecewise linear function $f_1(x)$ on $[0, M]$ that connects points $(0, 0)$, $(1+\eps/2, 0)$, $(1+\eps, 1)$, $(M, 1)$.
    This is an approximation to the counting function $\wb p$.
    We can rewrite $f_1$ as
    \begin{equation*}
        f_1(x) = \frac{2}{\eps}\Big(\relu(x - (1+\eps/2)) - \relu(x - (1+\eps))\Big).
    \end{equation*}
    By translating and scaling the domain to $[-1, 1]$ and applying \cref{lem:relu}, we can find $p_1$ of degree at most $\frac{CM}{\eps\eta}$ so that
    \begin{equation*}
        \|f_1 - p_1\|_{L^\infty([0, M])} \leq \frac{\eta}{2}.
    \end{equation*}
Define $f = (1- \eta) f_1 + (\eta/2)$ and $p = (1-\eta)p_1 + (\eta/2)$.  Then $f$ is the piecewise linear function on $[0,M]$ connecting points $(0,\eta/2)$, $(1+\eps/2, \eta/2)$, $(1+\eps, 1-(\eta/2))$, $(M, 1-(\eta/2))$, and 
    \begin{equation*}
        \|f - p\|_{L^\infty([0, M])} = (1-\eta)\| f_1 - p_1 \|_{L^\infty([0, M])} \leq \frac{\eta}{2}.
    \end{equation*}
    This implies that the 3 conditions hold for $p$ using the triangle inequality. 
\end{proof}

Next, we show how to use the constructed polynomial $p$ as a certificate to locate the median.

\begin{lemma}
    \label{lem:cerify}
    Let $V$ be the set to be sketched.
    Fix $\eps\in(0, 1)$ and $M \geq 3$.
    For $\eta \in (0, \frac{1}{3(k+1)}]$, take a norming set $Q = \set{\x_j}_j \in \cM\left(\overline{\pio(V)}, \frac{CM}{\eps\eta}, 1 + \eta\right)$.
    Let $\S_i$, $i = 1,\ldots,2k+1$, be a committee of sketch operators.
    Suppose the committee is such that $\max_{i,j}\|\S_i\x_j\|^2\leq M$.
    Then
    \begin{equation*}
        \med_i\|\S_i\x\|^2 \asymp 1 \pm \eps/2
    \end{equation*}
    for all $\x\in Q$ implies that
    \begin{equation*}
        \med_i\|\S_i\x\|^2 \asymp 1 \pm \eps
    \end{equation*}
    for all $\x\in \pio(V)$.
\end{lemma}

\begin{proof}
    We first prove one direction that if for all $\x\in Q$, $\med_i\|\S_i\x\|^2 \leq 1 + \eps/2$, then for all $\x \in \pio(V)$, $\med_i\|\S_i\x\|^2 \leq 1 + \eps$.
    To this end, we construct a polynomial $p$ on $[0, 2M]$ as in \cref{lem:approxcnt} of degree at most $CM/(\eps\eta)$ that satisfies the 3 conditions on $[0, 2M]$.
    Define a nonnegative valued polynomial of degree at most $2CM/(\eps\eta)$:
    \begin{equation*}
        P(\x) = \sum\nolimits_{i = 1}^{2k+1} p(\|\S_i\x\|^2).
    \end{equation*}
    Then we deduce for all $\x\in Q$,
    \begin{align*}
        \med_i\|\S_i\x\|^2 \leq 1 + \eps/2
        ~\Rightarrow~ 
        \#\set{i: \|\S_i\x\|^2 > 1 + \eps/2} \leq k
        ~\Rightarrow~
        P(\x) \leq (k+1)\eta + k.
    \end{align*}
    Since $Q$ is a norming set, this implies for all $\x\in \pio(V)$,
    \begin{equation*}
        P(\x) \leq (1+\eta)((k+1)\eta + k).
    \end{equation*}
    Since $\|\S_i\x\|^2\leq M$ on $Q$, $\|\S_i\x\|^2 \leq 2M$ on $\pio(V)$.
    Hence for all $\x\in\pio(V)$,
    \begin{equation*}
        \#\set{i: \|\S_i\x\|^2 > 1 + \eps} \leq \frac{1+\eta}{1-\eta}((k+1)\eta + k).
    \end{equation*}
    It is elementary to check that as long as $\eta \leq \frac{1}{3(k+1)}$ as required in the lemma statement, the right-hand side is strictly smaller than $k+1$.
    This implies $ \text{\small \#}\set{i: \|\S_i\x\|^2 > 1 + \eps} \leq k$ and hence 
    \begin{equation*}
    \med_i\|\S_i\x\|^2 \leq 1 + \eps.
    \end{equation*} 

    In order to show the other direction, we construct an approximate counting polynomial $r(x)$ on $[0, 2M]$ similar to $p(x)$, but
    (i) $r(x) \in [0, 1]$; (ii) $x\in[0, 1-\eps]$ implies $r(x) \geq 1-\eta$; and (iii) $x\in [1-\eps/2, 2M]$ implies $r(x) \leq \eta$.
    Repeating the argument in \cref{lem:approxcnt}, such $r$ can have the same degree as $p$.
    Applying the one-sided argument above to $R(\x) = \sum_{i=1}^{2k+1} r(\|\S_i\x\|^2)$ gives the bound in the other direction.
\end{proof}

The last ingredient needed to prove \cref{thm:medsketch} is a general form of guarantee for median of means estimation.

\begin{lemma}\label{lem:med}
    Let $X_1,\ldots, X_{2k+1}$ be i.i.d. copies of a random variable $X$.
    If for $a < b$ we have
    \begin{equation*}
        \pr(X \leq a) \leq p~~\text{and}~~\pr(X \geq b) \leq p,
    \end{equation*}
    then 
    \begin{equation*}
        \pr(\med_i X_i \notin [a, b]) \leq \frac{1}{\sqrt{\pi(k+\frac{1}{4})}}(4p)^{k+1}.
    \end{equation*}
\end{lemma}

\begin{proof} 
    Let the cdf of $X$ be $F$.
    By a union bound, the cdf of the median $G$ satisfies
    \begin{equation*}
        G(a) \leq \binom{2k+1}{k+1} F(a)^{k+1} = \binom{2k}{k}\frac{2k+1}{k+1}F(a)^{k+1}.
    \end{equation*}
    The binomial coefficient is related to the Catalan number and can be bounded by
    \begin{equation*}
        \binom{2k}{k} \leq \frac{4^k}{\sqrt{\pi(k + \frac{1}{4})}}.
    \end{equation*}
    We can lower bound $G(b)$ via a similar argument.
\end{proof}

Finally, we are ready to put everything together and prove \cref{thm:medsketch}.

\begin{proof}
     Fix $\eta = \frac{1}{3(k+1)}$. 
    We generate a norming set $Q \in\cM(\overline{\pio(V)}, \frac{CM}{\eps\eta}, 1 + \eta)$ of cardinality
    \begin{equation*}
        \log|Q| 
        \leq 
        {C}\log(2D) + C n_v\left[\log\left(n_v\cdot \frac{CM}{\eps\eta}\right) - \log\log(1 + \eta)\right]
        \leq
        {C}\log D + C n_v\log\left(\frac{n_vMk}{\eps}\right),
    \end{equation*}
    using \cref{thm:norming} and the fact that $\overline{\pio(V)} \subseteq \zcl{\pio(V_0)}$ where the dimension and degree of $\zcl{\pio(V_0)}$ are bounded by $n_v$ and $2D$ respectively (see the proof of \cref{thm:main}).
    According to \cref{lem:med}, for each $\x\in Q$, we have
    \begin{equation*}
        \pr(\med_i\|\S_i\x\|^2 \notin [1-\eps/2, 1+\eps/2]) \leq Ck^{-1/2} \exp(-C(\theta - \log 4)k).
    \end{equation*}
    Applying a union bound over $Q$, 
    \begin{multline*}
        \pr(\forall \x\in Q,~\med_i\|\S_i\x\|^2 \asymp 1\pm \eps/2) 
        \geq
        1 - C|Q|k^{-1/2}\exp(-C(\theta - \log 4)k)\\
        \geq
        1 - \exp\left[-C_1(\theta - \log 4)k + {C_2}\log D + C_3n_v \log\left(\frac{n_vMk}{\eps}\right)\right].
    \end{multline*}
    Since by assumption $\|\S_i\x\|^2 \leq M$ for all $\x\in Q$, \cref{lem:cerify} guarantees 
    \begin{multline*}
        \pr(\forall \x\in \pio(V),~\med_i\|\S_i\x\|^2 \asymp 1\pm \eps) \\
        \geq
        1 - \exp\left[-C_1(\theta - \log 4)k + {C_2}\log D + C_3n_v \log\left(\frac{n_vMk}{\eps}\right)\right].
    \end{multline*}
    Since $\med_i\|\S_i\pio(\x)\|^2 \asymp 1\pm \eps\,\Leftrightarrow\,\|\tilde S(\x)\|^2\asymp (1\pm\eps)\|\x\|^2$,
    this is what we wanted to show.
    
Under the additional assumptions, the above bound simplifies to
    \begin{equation*}
        \pr(\forall \x\in \pio(V),~\med_i\|\S_i\x\|^2 \asymp 1\pm \eps) 
        \geq
        1 - \exp\big(-Ck + Cn_v\log(M/\eps) + Cn_v\log k\big).
    \end{equation*}
    In order for this to be greater than $1-\delta$, we need
    \begin{equation*}
        k \geq Cn_v\log(M/\eps) + Cn_v\log k + C \log(1/\delta).
    \end{equation*}
    It is not hard to verify that when $M \geq d_v$ and
    \begin{equation*}
        k \geq C_4 (n_v \log(M/\eps) + \log(1/\delta))
    \end{equation*}
    for some constant $C_4$,
    the inequality indeed holds. 
\end{proof}

\begin{remark}\label{rmk:Mbnd-pf}
    As pointed out in \cref{rmk:Mbnd}, the condition $\|\S\x\|^2\leq M$ for all $\x\in\pio(V)$ is overkill.
    All we use in the proof is that for every $\S_i$ in the committee and every $\x_j$ in the norming set $Q$, $\|\S_i\x_j\|^2\leq M$ as required by \cref{lem:cerify}.
    Thus, it is possible to replace the boundedness condition by $\pr(\|\S\x\|^2\geq M) \leq e^{-\beta}$ for every fixed $\x\in\pio(V)$ and take a union \nolinebreak bound. 
\end{remark}

\vspace*{2ex}
\noindent
\bf{Acknowledgments.} 
Y.Z. was supported in part by a Graduate Continuing Fellowship and a National Initiative for Modeling and Simulation (NIMS) Graduate Student Fellowship at UT Austin. 
J.K. was supported in part by
NSF DMS 2309782, NSF DMS 2436499, NSF CISE-IIS 2312746, DE SC0025312, and a J. Tinsley Oden Faculty Fellowship at UT Austin.

\vspace*{2ex}
\printbibliography

\end{document}